\documentclass[a4paper,12pt]{amsart}

\usepackage[centertags]{amsmath}
\usepackage{amsfonts,amssymb,amstext,amsthm,newlfont,latexsym,stmaryrd}
\usepackage[left=3.0cm,right=3.0cm,top=3.7cm,bottom=3.7cm]{geometry}
\usepackage{tikz}
\usetikzlibrary{arrows,automata,positioning,fit,shapes}

\usepackage[T1]{fontenc}
\usepackage{lmodern}
\usepackage{microtype}
\usepackage[colorlinks=true,linkcolor=blue,citecolor=blue,urlcolor=blue]{hyperref}

\makeatletter
\@namedef{subjclassname@2020}{\textup{2020} Mathematics Subject Classification}
\makeatother

\theoremstyle{plain}
\newtheorem{theorem}{Theorem}[section]
\newtheorem{corollary}[theorem]{Corollary}
\newtheorem{lemma}[theorem]{Lemma}
\newtheorem{proposition}[theorem]{Proposition}
\theoremstyle{definition}
\newtheorem{definition}[theorem]{Definition}
\newtheorem{example}[theorem]{Example}
\theoremstyle{remark}
\newtheorem{remark}[theorem]{Remark}
\theoremstyle{definition}
\newtheorem*{ack*}{Acknowledgement}

\newcommand{\cB}{\mathcal B}
\newcommand{\cF}{\mathcal F}
\newcommand{\cP}{\mathcal P}
\newcommand{\cA}{\mathcal A}
\newcommand{\cM}{\mathcal M}
\newcommand{\sX}{\mathsf X}
\newcommand{\bN}{\mathbb N}
\newcommand{\bZ}{\mathbb Z}

\title[Bi-balanced shifts and variable specification]
      {Bi-balanced shifts and variable specification}
\author[S. Hong]{Soonjo Hong} \address{Hongik University \\
  2639, Sejong-ro, Jochiwon-eup 
  \\
  Sejong\\
  South Korea}
\email{hsoonjo@hongik.ac.kr}
\author[M. Kim]{Minkyu Kim} \address{Department of Mathematics\\
  Ajou University \\
  16499,
  Suwon\\
  South Korea}  
\email{angellucifer@ajou.ac.kr}
\date{}

\begin{document}

\begin{abstract}
Examples of one-sided balanced shifts which do not have variable specification are provided.
On the other hand, under synchronization condition, bi-balancedness is proven to be equivalent to variable specification.
\end{abstract}

\maketitle

\section{Introduction}

Baker and Ghenciu introduced (right) balancedness of a shift space as the existence of a uniform lower bound on the ratio of the number of followers to that of all words of a fixed length. \cite{BG15}.
Applying the property to predecessors we define left-balancedness, and in turn, bi-balancedness: it is just right- and left-balancedness at once.

It was shown that right-balancedness implies uniform exponential control of word growth \cite{BG15}. Moreover, they showed that balancedness is equivalent to the existence of a (not necessarily $\sigma$-invariant) Gibbs measure. Kim strengthened the conclusion by showing that, for a continuous potential $f$,  $f$-bi-balancedness is equivalent to the existence of an invariant Gibbs measure for $f$ \cite{KimMK24}.

On the other hand, specification property introduced by Bowen \cite{Bow71} expresses the possibility of joining finite pieces of orbits with a uniform gap.
In symbolic dynamics the same idea has a particularly concrete form:
any two admissible words can be joined by
a bridge of uniform length.

A survey \cite{KLO16} of Kwietniak, \L{}\k{a}cka and Oprocha introduces a broad panorama of specification-like properties and their dynamical consequences.
They typically force topological mixing and, in their periodic form, abundant periodic points and periodic measures.
They also give a rich structure to the invariant measure simplex:
ergodic measures are dense and invariant measures are approximated by periodic-orbit measures.
Among those consequences, one of the strongest results is intrinsic ergodicity \cite{Bow71,CliPav19}, that is, the uniqueness of measure of maximal entropy.

Comparing specification and balancedness, we see that in a sense bi-balancedness is a quantitative version of specification:
in a bi-balanced shift space we are able to join two sizable chunks of all admissible words freely.
Here a counting lower bound requirement replaces a geometric gap requirement for gluing.
Specifically, we see that, under synchronization condition, bi-balancedness is equivalent to variable specification.
The main result of the present paper is the following:

\begin{theorem}
  Let $X$ be an irreducible shift space.
  Then it is synchronized and bi-balanced if and only if it has variable specification.
\end{theorem}
Its proof is in section~\ref{section::equivalence}.
Via Kim's results \cite{KimMK24}, this implies that under synchronization condition the existence of invariant Gibbs measures is equivalent to variable specification and implies intrinsic ergodicity.
We will also see that variable specification induces intrinsic ergodicity.

\section{Preliminaries}

We assume basic knowledge of symbolic dynamics here.
More exposition about symbolic dynamics can be found in \cite{LinM95}.

Let $X\subset\cA^{\mathbb Z}$ be a shift space over a finite alphabet $\cA$ with the left shift map $\sigma$. 
Its language and the set of blocks of length $n$ are denoted by
\[
  \cB(X)=\bigcup_{n\geq0}\cB_n(X),
  \qquad \cB_n(X)=\cB(X)\cap\cA^n.
\]
where $\cB_0(X)$ is a singleton $\{\varepsilon\}$ consisting of the \emph{empty word} $\varepsilon$ of length $0$.
We say that $X$ is \emph{irreducible} if for any ordered pair $u,v$ of words in $\cB(X)$ there is a word $w$ in $\cB(X)$ with $uwv$ in $\cB(X)$.
Here $w$ is called a \emph{bridge} from $u$ to $v$.
In this paper, we will always assume that $X$ is irreducible.

The (\emph{topological}) \emph{entropy} of $X$ is defined to be
\[
h(X)=\lim_{n\to\infty}\frac1n\log |\cB_n(X)|.
\]
It measures the growth rate of the complexity of $\cB(X)$.

For $u$ in $\cB(X)$, let
\[ \cF(u,X)=\{w\in\cB_(X):uw\in\cB(X)\}\text{ and }\cP(u,X)=\{w\in\cB(X):wu\in\cB(X)\} \]
be the \emph{follower} and the \emph{predecessor} sets, respectively.
For $n\ge0$, let
\[ \cF_n(u,X)=\cF(u,X)\bigcap\cA^n,\text{ and }\cP_n(u,X)=\cP(u,X)\bigcap\cA^n. \]
We call $X$ \emph{right-balanced} if there is $c>0$ such that for all $u$ in $\cB(X),\ n\geq1$ we have
\[
  \frac{|\cF_n(u,X)|}{|\cB_n(X)|}\geq c,
\]
and \emph{left-balanced} if there is $c>0$ such that for all $u$ in $\cB(X),\ n\geq1$ we have
\[
\frac{|\cP_n(u,X)|}{|\cB_n(X)|}\geq c.
\]
It is \emph{bi-balanced} if it is both right- and left-balanced at the same time.

Let $C$ be a set of words over $\cA$.
Then a \emph{coded system} $\sX_C$ generated by $C$ is a shift space whose points are all bi-infinite concatenations of words from $C$.
A shift space $X$ is coded if $X=\sX_C$ for some $C\subset\cA^*$.
Denote by $C^*$ the set of all finite concatenations of words from $C$ including the empty word $\varepsilon$.
Every word of a coded system $\sX_C$ is a subword of some word in $C^*$.

A word $w$ in $\cB(X)$ is \emph{synchronizing} if whenever $uw,wv$ are in $\cB(X)$ we have $uwv\in\cB(X)$ as well.
An irreducible shift space containing such a word is called \emph{synchronized}.
Every word containing a synchronizing word as a subword is synchronizing as well and every synchronized shift is a coded system.

Finally,
$X$ has \emph{specification} if there is $N\geq0$ such that for every $u,v$ in $\cB(X)$ there is a bridge $w$ of length $N$ from $u$ to $v$,
and \emph{variable specification} if there is $N\geq0$ such that for every $u,v$ in $\cB(X)$ there is a bridge $w$ of length $\leq N$ from $u$ to $v$.
The notion of variable specification was introduced by Jung \cite{Jun11} under the name of \emph{almost specification}.
We will use the former term along with \cite{KLO16}.

\section{Balancedness and variable specification}

We show that one-sided balancedness alone cannot induce variable specification through concrete examples.

\begin{definition}
  A shift space $X$ is called \emph{boundedly supermultiplicative} (BSM for short) if there is a constant $c\ge1$ such that we have
  \[ |\cB_m(X)|\,|\cB_n(X)|\le c|\cB_{m+n}(X)| \]
  for all $m,n\in\mathbb{N}$.
\end{definition}

\begin{proposition}\cite{BG15}\label{prop::BSM_entropy_condition}
  A shift space $X$ is BSM if and only if there is a constant $c\ge1$ such that we have
  \[ \frac1c\le\frac{\exp(nh(X))}{|\cB_n(X)|}\le1 \]
  for all $n\in\mathbb{N}$.
\end{proposition}

Among BSM shifts, zero-entropy ones have simple structures.

\begin{proposition}\cite{KimThesis24}\label{prop::BSM_zero_entropy_condition}
  A zero-entropy shift space $X$ is BSM if and only if it consists of finitely many periodic points.
  Furthermore, it is irreducible and BSM if and only if it consists of a single periodic orbit.
\end{proposition}
\begin{proof}
  Let $X$ be a BSM shift space with $h(X)=0$.
  Then by Proposition~\ref{prop::BSM_entropy_condition} we have a constant $c\ge1$ such that for all $n$ in $\bN$ we have
  \[ |\cB_n(X)|\le c. \]
  So $X$ consists of periodic points.
  The remaining implications are trivial.
\end{proof}

It is known that a shift space with variable specification is right-balanced and that a right-balanced shift space is BSM \cite{BG15}.
Symmetrically, a shift with variable specification is left-balanced and a left-balanced shift space is BSM.
Immediately, a shift space with variable specification is bi-balanced.
Furthermore, by Propositions~\ref{prop::BSM_entropy_condition} and~\ref{prop::BSM_zero_entropy_condition}, the following are easily shown to be equivalent for a zero-entropy shift space $X$:
\begin{enumerate}
  \item $X$ has variable specification.
  \item $X$ is bi-balanced.
  \item $X$ is one-sided balanced.
  \item $X$ is BSM.
\end{enumerate}

$S$-gap shifts are another family of shift spaces where bi-balancedness and one-sided balancedness are equivalent.
An $S$-\emph{gap shift} $\sX(S)$ is defined as a coded system $\sX_C$ where $C=\{10^s\mid s\in S\},S\subset\bN,S\ne\emptyset$.
Every $S$-gap shift is BSM \cite{BG15}.

\begin{proposition}\cite{KimThesis24}\label{prop::S_gap_shift}
  Let $S=\{s_1,s_2,\cdots\}$ be a possibly finite subset of natural numbers.
  Then the following are equivalent for an $S$-gap shift $X=\sX(S)$:
  \begin{enumerate}
    \item $G_S=\{s_n-s_{n-1}\}$ is bounded.
    \item $X$ has variable specification.
    \item $X$ is bi-balanced.
    \item $X$ is one-sided balanced.
  \end{enumerate}
\end{proposition}
\begin{proof}
  The equivalence between the first two statements is shown in \cite{Jun11}.
  
  When $\sup G_S=\infty$, for each $k>0$ we have $n_k$ with $s_{n_k+1}-s_{n_k}>k$.
  Clearly $10^{s_{n_k}+1}$ and $0^{s_{n_k}+1}1$ have exactly one follower and exactly one predecessor of length $k$, respectively, both of which are just $0^k$.
  However, $S$ has infinitely many elements,
  so it is easy to verify that $X$ has positive entropy (see Proposition 2.1 of \cite{BG15}).
  Then $X$ is neither right nor left-balanced since both
  \[ \frac{|\cF_{k}(10^{s_{n_k}+1},X)|}{|\cB_k(X)|}=\frac{1}{|\cB_k(X)|}\text{ and }\frac{|\cP_{k}(0^{s_{n_k}+1}1,X)|}{|\cB_k(X)|}=\frac{1}{|\cB_k(X)|}\]
  go to 0.
  So, if $X$ is one-sided balanced, then $G_S$ is bounded, then $X$ has variable specification, then $X$ is bi-balanced.
\end{proof}

Until now, for zero-entropy shifts or $S$-gap shifts, bi-balanced shifts and one-sided balanced ones have coincided.
However, this is not true in general.
As Example~\ref{eg::right_not_left_balanced} and Theorem~\ref{thm::one-sided_not_bi-balanced} suggest, one-sided balancedness does not imply bi-balancedness, let alone variable specification.

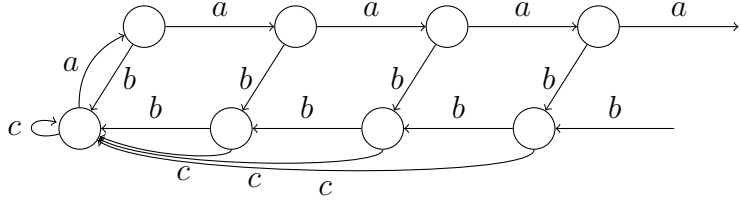
\begin{figure}[ht]
  \centering
  \begin{tikzpicture}[
    vertex/.style={
      circle,
      draw,
      inner sep=0pt,
      minimum size=5.5mm
    },
    edge/.style={->,thin}
    ]
    \def\N{4}
    
    \node[vertex] (p0) at (0,0) {};
    
    \foreach \i in {1,...,\N} {
      \pgfmathsetmacro{\x}{0.85+2*(\i-1)}
      \node[vertex] (q\i) at (\x,1.35) {};
    }
    
    \foreach \i in {1,...,3} {
      \pgfmathsetmacro{\x}{2*\i}
      \node[vertex] (p\i) at (\x,0) {};
    }
    
    \draw[edge]
    (p0) edge[loop left] node[left] {$c$} (p0);
    
    \draw[edge]
    (p0) to[bend left=34] node[left] {$a$} (q1);
    \draw[edge]
    (q1) to node[right] {$b$} (p0);
    
    \foreach \i in {1,...,3} {
      \pgfmathtruncatemacro{\j}{\i+1}
      \draw[edge]
      (q\i) -- node[above] {$a$} (q\j);
    }
    \draw[edge]
    (q4) -- node[above] {$a$} ++(1.85,0);
    
    \foreach \i in {2,...,4} {
      \pgfmathtruncatemacro{\j}{\i-1}
      \draw[edge]
      (q\i) -- node[left] {$b$} (p\j);
    }
    
    \foreach \i in {1,...,3} {
      \pgfmathtruncatemacro{\j}{\i-1}
      \draw[edge]
      (p\i) -- node[above] {$b$} (p\j);
    }
    \draw[edge]
    (7.85,0) -- node[above] {$b$} (p3);
    
    \foreach \i in {1,...,3} {
      \pgfmathsetmacro{\depth}{0.28+0.13*\i}
      \draw[edge]
      (p\i)
      .. controls +(0,-\depth)
      and +(1.00,-\depth) ..
      node[pos=.52,below] {$c$}
      (p0);
    }
  \end{tikzpicture}
  \caption{A labeled graph for
    Example~\ref{eg::right_not_left_balanced}.}
  \label{fig::right_not_left_balanced}
\end{figure}

\begin{example}\label{eg::right_not_left_balanced}
  Let
  \[ C=\{a^nb^n,a^lb^mc,c\mid l,m,n\in\mathbb{N},l>m\} \]
  and $X=\mathsf{X}_C$.
  It is a coded system presented by the graph described in Figure~\ref{fig::right_not_left_balanced} and clearly has positive entropy, say $\log\lambda,\lambda>1$.
  We will show that $X$ is right- but not left-balanced.
  Immediately, it cannot satisfy variable specification as such shifts must be bi-balanced.
  
  Note that $\cP_n(ab^{n+1},X)$ is a singleton $\{a^n\}$.
  Then
  \[ \frac{|\cP_n(ab^{n+1},X)|}{|\cB_n(X)|}\le\frac{1}{\lambda^{n}}\to0 \]
  as $n\to\infty$,
  so $X$ is not left-balanced.
  
  To see that $X$ is right-balanced, we will estimate $|\cF_n(u,X)|$ and $|\cB_n(X)|$ for all $n$ in $\bN$ and $u$ in $\cB(X)$.
  Every edge in Figure~\ref{fig::right_not_left_balanced} has exactly two outgoing edges of distinct labels,
  so every word in $X$ has at least $2^n$ followers of length $n$.
  It follows that for every word $u$ of $X$ we have
  \[ |\cF_n(u,X)|\ge2^n. \]
  
  Now, we estimate the cardinality of all words of length $n$.
  First, every word containing $c$ or $ba$ has a single terminal vertex in the graph.
  Hence, it has exactly $2^n$ followers of length $n$.
  Next, since $b$ is followed by $a,b$ and $c$,
  we have a disjoint union
  \[ \cF_n(b,X)=\bigcup_{s\in\{a,b,c\}}s\cF_{n-1}(bs,X) \]
  where $s\cF_{n-1}(bs,X)=\{sw\mid w\in\cF_{n-1}(bs,X)\}$ for $s=a,b$ or $c$.
  Then
  \[\begin{split}
    |\cF_n(b,X)|=&|\cF_{n-1}(ba,X)|+|\cF_{n-1}(bb,X)|+|\cF_{n-1}(bc,X)|\\
    =&2^{n-1}+|\cF_{n-2}(bba,X)|+|\cF_{n-2}(bbb,X)|+|\cF_{n-2}(bbc,X)|+2^{n-1}\\
    =&2\cdot2^{n-1}+2^{n-2}+|\cF_{n-3}(b^3a,X)|+|\cF_{n-3}(b^4,X)|+|\cF_{n-3}(b^3c,X)|+2^{n-2}\\
    &\vdots\\
    =&2(2^{n-1}+2^{n-2}+\cdots+2)+|\cF_1(b^n,X)|\\
    =&2^{n}+2^{n-1}+\cdots+4+|\{a,b,c\}|\\
    =&2^{n+1}-1.
  \end{split}\]
  Also, $a$ is followed by $a$ and $b$ but not by $c$.
  We have
  \[\begin{split}
    |\cF_n(a,X)|=&|\cF_{n-1}(aa,X)|+|\cF_{n-1}(ab,X)|\\
    =&|\cF_{n-2}(aaa,X)|+|\cF_{n-2}(aab,X)|+|\cF_{n-1}(ab,X)|\\
    &\vdots\\
    =&|\cF_{1}(a^n,X)|+\sum_{k=1}^{n-1}|\cF_{n-k}(a^kb,X)|\\
    \le&|\{a,b\}|+\sum_{k=1}^{n-1}|\cF_{n-k}(b,X)|\\
    =&2+\sum_{k=1}^{n-1}(2^{n-k+1}-1)\\
    =&2^{n+1}-n-1.
  \end{split}\]
Finally,
\[\begin{split}
  |\cB_n(X)|=&|\cF_{n-1}(a,X)|+|\cF_{n-1}(b,X)|+|\cF_{n-1}(c,X)|\\
  \le&2^n-n+2^n-1+2^{n-1}\\
  \le&5\cdot2^{n-1}.
\end{split}\]
Therefore, we have
\[ \frac{|\cF_n(u,X)|}{|\cB_n(X)|}\ge\frac{2^n}{5\cdot2^{n-1}}=\frac25 \]
showing that $X$ is right-balanced.
\end{example}

By Theorem~\ref{thm::one-sided_not_bi-balanced}, we can construct more examples of one-sided but not bi-balanced shifts.

\begin{lemma}\cite{Jun11,Tho04}\label{lmm::one-sided_not_bi-balanced}
  Let $X\subset\cA^\bZ$ be a shift space with variable specification.
  Then there exists $p$ in $\bN$ and closed subsets $X_0,X_1,\cdots,X_{p-1}$ of $X$ such that
  \begin{enumerate}
    \item $X=\bigcup_{j=0}^{p-1}X_j$.
    \item $\sigma(X_j)=X_{j-1\mod{p}}$.
    \item $X_j\cap X_k,j\ne k$ has empty interior.
    \item $(X_j^p,\sigma),j=0,1,\cdots,p-1$, have specification where
    \[ X_j^p=\{\bar x\in(\cA^p)^\bZ\mid\text{there is }x\in X_j\text{ with }\bar x_i=x_{pi}x_{pi+1}\cdots x_{p(i+1)-1},i\in\bZ\}. \]
  \end{enumerate}
\end{lemma}

\begin{theorem}\cite{KimThesis24}\label{thm::one-sided_not_bi-balanced}
For any $h>0$ there is a coded system $X$ satisfying the following properties:
  \begin{enumerate}
    \item $h(X)=h$.
    \item $X$ is one-sided but not bi-balanced.
    Hence it has no variable specification.
    \item $X$ contains a subshift $Y$ with variable specification and $h(Y)=h$.
  \end{enumerate}
\end{theorem}
\begin{proof}
  Let $\lambda=e^h>1$.
  For $0<h<1$ it was shown in \cite{BG15} that there exists an $S$-gap shift $Y$ with $h(Y)=\log\lambda=h$ and specification.
  For general $h>0$, by taking a product shift space of an $S$-gap shift and $\{0,1,\cdots,n-1\}^{\bZ}$ for some $n\ge1$,
  we can find a shift space $Y$ with $h(Y)=h$ and variable specification
  By renaming the symbols if needed, we may assume that the alphabet of $Y$ does not include the symbol $\bar1$.
  
  We will define a set $C$ of words which gives a coded system $X=\sX_C$.
  Apply Lemma~\ref{lmm::one-sided_not_bi-balanced} to obtain subshifts $Y_j,0\le j<p$ of $Y$ satisfying the listed four properties.
  Let $M_j$ be a specification gap for each $(Y_j)^p,0\le j<p$, and $M=\max_jM_j$.
  Since $h(Y)=\log\lambda$ and $Y$ is BSM, by Proposition~\ref{prop::BSM_entropy_condition} there is $c_Y\ge1$ such that for all $n$ in $\bN$ we have $|\cB_n(Y)|\le c_Y\lambda^n$.
  Fix an integer $K>pM$ with $\lambda^K>c_Y$,
  and an infinite subset $S$ of $\{s\in p\bN\mid s>K\}$ with $h(\sX(S))=\log\eta<\log\lambda$.
  Define unbounded sequences $\{a_s\}_{s\in S},\{b_s\}_{s\in S}$ of positive multiples of $p$ such that
  \[ a_s+b_s+pM=s\in S,\alpha=\sup_{s\in S}\frac{a_s+K}{s}<1-\frac{\log\eta}{\log\lambda}\text{ and }b_s\text{ is increasing.} \]
  Those sequences are easily found if $S$ is sparse and each $a_s$ is much smaller than $s$.
  (Note that the sparser $S$ is, the smaller $\eta$ is, so $a_s$ is easier to satisfy the desired inequality.)
  As an $S$-gap shift, $\sX(S)$ is BSM \cite{BG15}.
  Then there is $c_S\ge1$ such that for all $n$ in $\bN$ we have $|\cB_n(\sX(S))|\le c_S\eta^n$.
  
  Fix $x^{(j)}$ in $Y_j$ for each $0\le j<p$ and let $v_{j,k}=x^{(j)}_{[-k,-1]},k\in\bN$. 
  Regard each $u$ in $\cB_{a_s}(Y)$ as a word of length $a_s/p$ in $(Y_{j_u})^p$ for some $0\le j_u<p$,
  and find $w_{u,b_s}$ in $\cB_{pM}(Y)$ with
  \[ uw_{u,b_s}v_{j_u,b_s}\in\cB_s(Y). \]
  Such $w_{u,b_s}$ exists since $(Y_{j_u})^p$ has specification with gap $M$.
  A bridge of length $M$ from $u$ to $v_{j_u,b_s}$ in $(Y_{j_u})^p$ is a bridge of length $pM$ in $Y_{j_u}$.
  
  Now, let
  \[ C=\{\bar1uw_{u,b_s}v_{j_u,b_s}\mid u\in\cB_{a_s}(Y),s\in S\} \]
  and $X=\sX(C)$.
  Then $\bar1$ is a synchronizing word of $X$.
  
  We will first compute $h(X)$.
  Since $\{a_s\}$ is unbounded, every word of $Y$ is a subword of some $u$ in $\cB_{a_s}(Y)$.
  Thus we have $\cB(Y)\subset\cB(X)$ and $h(X)\ge h(Y)=h$.
  To show the other inequality, define a $1$-block map $\pi:X\to\sX(S)$ by
  \[ \pi(\bar1)=1,\pi(a)=0\text{ for }a\ne\bar1. \]
  Every word in $\pi^{-1}(10^s1)$ is of the form $\bar1w\bar1$ where $\bar1$ does not appear in $w$.
  So $\bar1w$ must belong to $C$ and $w$ is a word of length $s$ determined by some $u$ in $\cB_{a_s}(Y)$.
  It follows that
  \[ |\pi^{-1}(10^s1)|=|\cB_{a_s}(Y)|\le c_Y\lambda^{a_s}\le\lambda^{K+a_s}\le\lambda^{s\alpha}. \]
  Since $\{a_s\}$ is unbounded, we have
  \[\begin{split}
    \pi^{-1}(10^m)=&\bar1\cB_{m}(Y),\\
    |\pi^{-1}(10^m)|=&|\cB_{m}(Y)|\le c_Y\lambda^{m}.
  \end{split}\]
 
  We turn to estimating $|\pi^{-1}(0^l1)|$.
  Let $s_0=\min\{s\in S\mid s\ge l\}$ and $b^*=b_{s_0}$.
  In the case $l\le b^*$, the $0^l$ in $0^l1$ is always lifted by $\pi^{-1}$ to a suffix of $v_{j,b^*}$ for some $0\le j<p$.
  This is because $\{b_s\}$ is unbounded and increasing, so $v_{j_u,b_s}$ is a suffix of $v_{j_u,b_{s'}}$ for every $s'>s$:
  even if that $0^l$ comes from a suffix of $u_{a_s}w_{u,b_s}v_{j_u,b_s}$ for some $s>s_0$, the suffix is also a suffix of $v_{j_u,b^*}$.
  Hence $|\pi^{-1}(0^l1)|$ is no greater than the number, which is just $p$, of $v_{j,b^*}$.
  In the case $l>b^*$, the $0^{b^*}$ right before $1$ is always lifted by $\pi^{-1}$ to $v_{j,b^*}$ for some $0\le j<p$,
  and the preimages of the leftmost $0^{l-b^*}$ in $0^l1$ are no more than $|\cB_{l-b^*}(Y)|$.
  Then we obtain
  \[\begin{split}
    \alpha s_0\ge&a_{s_0}+K\\
    b^*=&s_0-a_{s_0}-pM\\
    \ge&s_0-(\alpha s_0-K)-pM\\
    \ge&(1-\alpha)l+(K-pM)\\
    l-b^*\le&l\alpha-(K-pM)\le l\alpha\\
    |\cB_{l-b^*}(Y)|\le&c_Y\lambda^{n-b^*}\le c_Y\lambda^{l\alpha}\\
    |\pi^{-1}(0^l1)|\le&(\#\text{ of }v_{j,b^*})\cdot|\cB_{l-b^*}(Y)|\le pc_Y\lambda^{l\alpha}
  \end{split}\]
  In either case $l\le b^*$ or $l>b^*$, we get
  \[ |\pi^{-1}(0^l1)|\le pc_Y\lambda^{l\alpha} \]
  
  Since $\bar1$ is synchronizing for $X$, for a word $w=0^l10^{s_1}10^{s_2}\cdots10^{s_k}10^m,l,m\ge0,s_j$ in $S$, of length $n$ in $\cB(\sX(S))$ we have
  \[\begin{split}
    |\pi^{-1}(w)|\le&|\pi^{-1}(0^l1)|\,|\pi^{-1}(10^{s_1}1)|\,|\pi^{-1}(10^{s_2}1)|\cdots|\pi^{-1}(10^{s_k}1)|\,|\pi^{-1}(10^m)|\\
    \le& pc_Y\lambda^{l\alpha}\lambda^{s_1\alpha}\lambda^{s_2\alpha}\cdots\lambda^{s_k\alpha}c_Y\lambda^m\\
    \le&pc_Y^2\lambda^{(n-m)\alpha+m}
  \end{split}\]
  
  Denote by $\cB_n(\sX(S),k)$ the set of all words in $\cB_n(\sX(S))$ in which the last appearance of $1$ is at the $k$-th coordinate.
  For every word $w$ in  $\cB_n(\sX(S),k)$ its suffix of length $(n-k)$ is fixed to be $0^{n-k}$ so
  \[ |\pi^{-1}(w)|\le pc_Y^2\lambda^{k\alpha+(n-k)}. \]
  Then
  \[\begin{split}
    |\cB_n(X)|=&\sum_{k=1}^n\sum_{w\in\cB_n(\sX(S),k)}|\pi^{-1}(w)|+|\{0^n\}|\,|\pi^{-1}(0^n)|\\
    \le&\sum_{k=1}^n|\cB_n(\sX(S),k)|\,pc_Y^2\lambda^{k\alpha+(n-k)}+|\cB_n(Y)|\\
    \le&\sum_{k=1}^nc_S\eta^k\cdot pc_Y^2\lambda^{k\alpha+(n-k)}+c_Y\lambda^n\\
    \le&pc_Y^2c_S\lambda^n\sum_{k=1}^n\eta^{k}\lambda^{k\alpha-k}+c_Y\lambda^n\\
    \le&pc_Y^2c_S\lambda^n\sum_{k=0}^\infty\left(\lambda^{\frac{\log\eta}{\log\lambda}+\alpha-1}\right)^k+c_Y\lambda^n\\
    \le&c\lambda^n
  \end{split}\]
  where $c=pc_Y^2c_SL+c_Y$ and $L=\sum_{k=0}^\infty\left(\lambda^{\frac{\log\eta}{\log\lambda}+\alpha-1}\right)^k$.
  (Note that $L$ converges since $\lambda>1$ and $\alpha<1-\dfrac{\log\eta}{\log\lambda}$.)
  This shows $h(X)\le\log\lambda$ and completes the proof of $h(X)=h$.
  
  Next we prove the right-balancedness of $X$.
  Let $u$ belong to $\cB(X)$.
  In the case that $u$ contains $\bar1$, write $u=v_1\bar1v_2\bar1\cdots\bar1v_k$ where $v_1,\cdots,v_k$ are in $\cB(Y)$.
  Then $\cF_n(u,X)=\cF_n(\bar1v_k,X)$ since $\bar1$ is synchronizing.
  Also $\cF_n(\bar1v_k,X)$ contains $\cF_n(v_k,Y)$ as $\{a_s\}$ is unbounded.
  A shift $Y$ with variable specification is already right-balanced, hence there is $0<c_r\le1$ with
  \[ \frac{|\cF_n(v_k,Y)|}{|\cB_n(Y)|}\ge c_r. \]
  It follows that
  \[\begin{split}
    \frac{|\cF_n(u,X)|}{|\cB_n(X)|}\ge\frac{|\cF_n(v_k,Y)|\,/\,|\cB_n(Y)|}{|\cB_n(X)|\,/\,|\cB_n(Y)|}
    \ge \frac{c_r}{c\lambda^n/\lambda^n}
    =\frac{c_r}{c}.
  \end{split}\]
  In the other case, $u$ avoids $\bar1$ and belongs to $\cB(Y)$.
  Again $\cF_n(u,X)$ contains $\cF_n(u,Y)$ as $\{a_s\}$ is unbounded.
  Hence, we see
  \[ \frac{|\cF_n(u,Y)|}{|\cB_n(Y)|}\ge c_r \]
  and
  \[\begin{split}
    \frac{|\cF_n(u,X)|}{|\cB_n(X)|}\ge\frac{|\cF_n(u,Y)|\,/\,|\cB_n(Y)|}{|\cB_n(X)|\,/\,|\cB_n(Y)|}
    \ge\frac{c_r}{c}.
  \end{split}\]
  Therefore $X$ is right-balanced.
  
  We show that $X$ is not left-balanced by contradiction.
  Suppose that $X$ is left-balanced and that there is $0<c_l\le1$ such that for all $u$ in $\cB(X)$ we have
  \[ \frac{|\cP_n(u,X)|}{|\cB_n(X)|}\ge c_l. \]
  Denote by $\cB_n(X,k)$ the set of all words in $\cB_n(X)$ in which the last appearance of $\bar1$ is at the $k$-th coordinate.
  The last $\bar1$ at the $k$-th coordinate is synchronizing and followed by any words in $\cB(Y)$, so for $1\le k\le n$ we have
  \[\begin{split}
    |\cB_n(X,k)|=&|\cP_{k-1}(\bar1)|\,|\cB_{n-k}(Y)|\\
    \ge &c_l|\cB_{k-1}(X)|\,|\cB_{n-k}(Y)|\\
    \ge&c_l|\cB_{k-1}(Y)|\,|\cB_{n-k}(Y)|\ge\frac{c_l}{\lambda}\lambda^{n}.
  \end{split}\]
  Considering the words of $\cB(X)$ which avoid $\bar1$, we see
  \[\begin{split}
    |\cB_n(X)|=&\sum_{k=1}^n|\cB_n(X,k)|+|\cB_n(Y)|\\
    \ge&\sum_{k=1}^n\frac{c_l}{\lambda}\lambda^{n}+\lambda^n=\lambda^n\left(\frac{nc_l}{\lambda}+1\right),
  \end{split}\]
  contradicting $|\cB_n(X)|\le c\lambda^n$, for large enough $n$.
  Hence $X$ is not left-balanced, and has no variable specification.
\end{proof}

\section{Bi-balancedness implies variable specification}\label{section::equivalence}
Variable specification implies bi-balancedness \cite{BG15} and synchronization \cite{Jun11}.
In this section, we will prove its converse to set up an equivalence condition.
We need some lemmata.

\begin{definition}
  We call a shift space $X$ \emph{entropy minimal} if every proper subshift $Y$ of $X$
  satisfies $h(Y)<h(X)$.
\end{definition}

We will show that synchronized and bi-balanced shift spaces are entropy minimal.
For that, the notions of Gibbs measure and measure-theoretic entropy are required:
let $w$ be a word of a shift space $X$ and $n$ an integer.
A subset of $X$ denoted by
\[ [w]_n=\{x\in X\mid x_{[n,n+|w|-1]}=w\}\]
is called a \emph{cylinder set} of $X$.
\begin{definition}\label{defn::measure-theoretic-entropy}
  Let $\cM(X)$ be the set of $\sigma$-invariant Borel probability measures on a shift space $X$.
  Given $\mu$ in $\cM(X)$, we define \emph{measure-theoretic entropy} $h_\mu(X)$ by
\[ h_\mu(X)=-\lim_{n\to\infty}\frac1n\sum_{w\in\cB_n(X)}\mu([w]_0)\log\mu([w]_0). \]
\end{definition}

Measure-theoretic entropy has affinity in the sense that if $\mu=t\nu+(1-t)\eta$ for $t$ in $(0,1)$ and $\sigma$-invariant Borel probability measures $\nu,\eta$ on $X$ then we have (see \cite[Proposition 10.13]{DGS76})
\[ h_\mu(X)=th_\nu(X)+(1-t)h_\eta(X). \]

\begin{definition}\label{defn::mme}
The well-known variational principle states that
\[ h(X)=\sup_{\mu\in\cM(X)}h_\mu(X) \]
and that there is $\mu$ on $X$ attaining the supremum \cite{DGS76,LinM95}.
Such $\mu$ is called a \emph{measure of maximal entropy} (MME for short).
When a shift space $X$ admits only one MME, $X$ is called \emph{intrinsically ergodic}.
\end{definition}
\begin{definition}\label{defn::gibbs}
  A $\sigma$-invariant probability measure $\mu$ on a shift space $X$ is called an (\emph{invariant}) \emph{Gibbs measure} if there is a \emph{Gibbs constant} $K>0$ satisfying the \emph{Gibbs inequality}
  \[ \frac1K\leq \mu([w]_0)e^{n h(X)}\leq K \]
  for all $w$ in $\cB_n(X)$.
\end{definition}

It is known that a shift space with specification is intrinsically ergodic.
We will show later in this section that a shift space with variable specification is intrinsically ergodic.

The next observation is stated by Thomsen in the paragraph immediately following Corollary 6.8 in \cite{THOM06}.
\begin{lemma}\cite{THOM06}\label{lmm::synchronized_fully-supported_mme}
  A synchronized shift has at most one MME with full support.
\end{lemma}

\begin{lemma}\label{lmm::erogidc_gibbs_entropy-minimal}
  A shift space equipped with an ergodic Gibbs measure is entropy minimal.
\end{lemma}
\begin{proof}
  Let $\mu$ be an ergodic Gibbs measure on a shift space $X$.
  We prove the entropy minimality of $X$ by contradiction.
  Suppose that $X$ is not entropy minimal.
  There is a proper subshift $Y\subsetneq X$ with $h(Y)=h(X)$.
  For all $n\ge1$ let
  \[ E_n=\bigcup_{w\in\cB_{2n+1}(Y)}[w]_{-n}. \]
  Then $E_n$ is decreasing as $n$ grows and $Y=\bigcap_{n\ge1}E_n$.
  
  By Gibbs inequality,
  \[ \mu(E_n)=\sum_{w\in\cB_{2n+1}(Y)}\mu([w]_{-n})\ge|\cB_{2n+1}(Y)|\frac{e^{-h(2n+1)}}K \]
  where $K$ is a Gibbs constant and $h=h(X)$.
  Since $|\cB_{2n+1}(Y)|\ge e^{h(2n+1)}$, we have
  \[ \mu(E_n)\ge\frac{1}K\text{ and }\mu(Y)=\lim_{n\to\infty}\mu(E_n)\ge\frac1K>0. \]
  Also, for any $w$ in $\cB(X)\setminus\cB(Y)$, again by Gibbs inequality, we obtain
  \[ \mu(Y)\le1-\mu([w]_0)\le1-\frac{e^{-h|w|}}{K}<1. \]
  Then $0<\mu(Y)<1$, which contradicts the ergodicity of $\mu$.
  Hence there is no such $Y$ and $X$ must be entropy minimal.
\end{proof}

\begin{proposition}\label{prop::synchronized_bi-balanced_entropy-minimal}
  A synchronized and bi-balanced shift $X$ is entropy minimal.
\end{proposition}
\begin{proof}
  Kim showed in \cite{KimMK24} that a shift space is bi-balanced if and only if it has an invariant Gibbs measure $\mu$ and that in such a case the measure is a fully supported MME on it.
  As $X$ is synchronized, $\mu$ is the unique fully supported MME of $X$ by Lemma~\ref{lmm::synchronized_fully-supported_mme}.
  
  We claim that $\mu$ is ergodic.
  It is a little modification of a standard result that the unique MME is ergodic.
  Suppose not and let $A$ be a $\sigma$-invariant $\mu$-measurable subset of $X$ with
  \[ 0<t=\mu(A)<1. \]
  Define $\mu_A(B)=\dfrac{\mu(B\cap A)}{t}$ and $\mu_{A^C}(B)=\dfrac{\mu(B\cap A^C)}{1-t}$.
  Both are invariant and $\mu=t\mu_A+(1-t)\mu_{A^C}$.
  So
  \[ h(X)=h_\mu(X)=th_{\mu_A}(X)+(1-t)h_{\mu_{A^C}}(X). \]
  Since $h_{\mu_A},h_{\mu_{A^C}}\le h(X)$, we have
  \[ h(X)=h_\mu(X)=h_{\mu_A}(X)=h_{\mu_{A^C}}(X). \]
  Then $\mu_s=s\mu_A+(1-s)\mu_{A^C}$ is also a fully supported MME of X for every $s$ in $(0,1)$, contradicting the uniqueness.
  Hence there is no such $A$, and $\mu$ is ergodic.
  Now the previous Lemma~\ref{lmm::erogidc_gibbs_entropy-minimal} ensures the entropy minimality of $X$.
\end{proof}

\begin{corollary}\label{cor::synchronized_bi-balanced_intrinsically-ergodic}
  A synchronized and bi-balanced shift is intrinsically ergodic.
\end{corollary}
\begin{proof}
  Let $X$ be a synchronized and bi-balanced shift.
  Then $X$ is entropy minimal and has the unique MME $\mu$ of full support.
  Another MME $\nu$ of $X$, if it exists, must have smaller support than $X$, say $Y\subsetneq X$.
  But, the entropy minimality of $X$ together with the variational principle forces
  \[ h_\nu(Y)\le h(Y)<h(X) \]
  denying that $\nu$ is an MME of $X$.
\end{proof}

Let $\cB$ be a set of words over a finite alphabet $\cA$ and $\cB_n=\cB\cap\cA^n$.
The entropy $h(\cB)$ of $\cB$ is defined by
\[ h(\cB)=\limsup_{n\to\infty}\frac1n\log|\cB_n|.\qquad(-\infty\text{ allowed}.) \]
We say that $\cB$ is \emph{factorial} if, given any word $u$ in $\cB$, all the subwords of $u$ belong to $\cB$ as well.
When $\cB$ is factorial, define a shift space $\sX(\cB)$ by
\[ x\text{ is in }\sX(\cB)\iff\text{every subword of }x\text{ is in }\cB. \]
It was shown in Lemma 2.7 of \cite{CliPav19} that $h(\sX(\cB))=h(\cB)$.

\begin{lemma}\label{lmm::entropy-minimal_bi-balanced_followers_predecessors}
Let $X$ be entropy minimal and bi-balanced.
Fix any $v$ in $\cB(X)$.
Then, there is $N=N(v)$ such that, for every $u$ in $\cB(X)$,
both $\cF_N(u,X)$ and $\cP_N(u,X)$ contain words $s$ and $p$, respectively, in which $v$ occurs as a subword.
\end{lemma}
\begin{proof}
As $X$ is bi-balanced, there is $c>0$ with
\[ \frac{|\cF_N(u,X)|}{|\cB_N(X)|},\frac{|\cP_N(u,X)|}{|\cB_N(X)|}>c. \]
Let
\[ D=\{w\in\cB(X)\mid v\text{ is not a subword of }w\}\text{ and }D_n=D\cap\cA^n. \]
Then $D$ is factorial and
\[ \sX(D)=\{x\in X\mid v\text{ does not appear in }x\}\subsetneq X. \]
Consequently, $h(\sX(D))=h(D)<h(X)$ by entropy minimality of $X$.

Choose $\alpha$ with $h(D)<\alpha<h(X)$.
There is some $M>0$ such that for all $n\ge M$ we have
\[ |D_n|<e^{n\alpha}<e^{nh(X)}\le|\cB_n(X)|. \]
Then
\[ \frac{|D_n|}{|\cB_n(X)|}<e^{n(\alpha-h(X))}\to0\]
so there exists $N>0$ with $|D_N|\,/\,|\cB_N(X)|<c$.
Suppose that for a word $u$ in $\cB(X)$ no word in $\cF_N(u,X)$ contains $v$ as a subword.
Immediately, $\cF_N(u,X)\subset D_N$, which induces
a contradiction
\[ c\le\frac{|\cF_N(u,X)|}{|\cB_N(X)|}\le\frac{|D_N|}{|\cB_N(X)|}<c. \]
 It follows that for any $u$ in $\cB(X)$ there is a word in $\cF_N(u,X)$ containing $v$ as its subword.
 Symmetrically, for any $u$ in $\cB(X)$ there is a word in $\cP_N(u,X)$ containing $v$ as its subword.
\end{proof}

\begin{theorem}\label{thm::main_equivalence}
Let $X$ be an irreducible shift space.
Then it is synchronized and bi-balanced if and only if it has variable specification.
\end{theorem}
\begin{proof}
  $\Leftarrow$): We refer to \cite{BG15} and \cite{Jun11}.
  
  $\Rightarrow$): By Proposition~\ref{prop::synchronized_bi-balanced_entropy-minimal} $X$ is entropy minimal.
  Choose a synchronizing word $v$ of $X$.
  Apply Lemma~\ref{lmm::entropy-minimal_bi-balanced_followers_predecessors} to $v$, obtaining $N=N(v)$.
  Given arbitrary $u,w\in\cB(X)$ choose a word $\alpha v$ in $\cF(u)$ and a word $v\beta$ in $\cP(w)$ with $|\alpha v|,|v\beta|\le N$.
  Since $v$ is synchronizing, we have $u\alpha v\beta w\in\cB(X)$.
  The bridge $\alpha v\beta$ has length at most $2N$, proving variable specification.
\end{proof}

The next corollary is an immediate result of Corollary~\ref{cor::synchronized_bi-balanced_intrinsically-ergodic} applied to Theorem~\ref{thm::main_equivalence}.
\begin{corollary}\label{cor:variable-specification_intrinsic-ergodicity}
  A shift with variable specification is intrinsically ergodic.
\end{corollary}

\begin{remark}[Open problem]
Does entropy minimality together with bi-balancedness force synchronization?
An affirmative answer would yield
\[
  \text{entropy minimality}+\text{bi-balancedness}\quad\Longrightarrow\quad\text{variable specification}
\]
\end{remark}

\bibliographystyle{amsplain}
\bibliography{ref}
\end{document}